\documentclass[a4paper]{amsart}

\pdfoutput=1

\usepackage[utf8]{inputenc}
\usepackage[T1]{fontenc}
\usepackage{lmodern}
\usepackage{amsthm, amssymb, amsmath, amsfonts, mathrsfs}

\usepackage{microtype}

\usepackage[pagebackref,colorlinks=true,pdfpagemode=none,urlcolor=blue,
linkcolor=blue,citecolor=blue]{hyperref}

\usepackage{amsmath,amsfonts,amssymb,amsthm}
\usepackage{amsthm, amssymb, amsmath, amsfonts, mathrsfs}

\usepackage{mathrsfs}
\usepackage{MnSymbol}
\usepackage{scalerel} 

\usepackage{color}
\usepackage{accents}

\usepackage{bbm}

\usepackage{dsfont}

\newtheorem{theorem}{Theorem}[section]
\newtheorem{lemma}[theorem]{Lemma}

\newtheorem{proposition}[theorem]{Proposition}

\numberwithin{equation}{section}
\numberwithin{figure}{section}
\theoremstyle{plain}

  \theoremstyle{definition}
  
  \theoremstyle{plain}
  
  \theoremstyle{plain}

\theoremstyle{remark}
\newtheorem{remark}[theorem]{Remark}

\renewenvironment{proof}[1][Proof]{ {\itshape \noindent {#1.}} }{$\Box$
\medskip}

\numberwithin{equation}{section}
\newcommand{\R}{\mathbb{R}}

\newcommand{\Z}{\mathbb{Z}}

\newcommand{\E}{\mathbb{E}}

\newcommand{\U}{\mathcal{U}}

\newcommand{\eps}{\varepsilon}
\def\les{\lesssim}

\newcommand{\Var}{\mathrm{Var}}

\newcommand{\la}{\langle}
\newcommand{\ra}{\rangle}

\newcommand{\cov}{\mathrm{Cov}}

\newcommand{\dd}{\mathrm{d}}
\newcommand{\sigL}{\sigma_{\mathrm{Lip}}}

\newcommand{\blue}{\textcolor{blue}}

\begin{document}

\title[Fluctuations of nonlinear SHE in $d\geq 3$]{Fluctuations of a nonlinear stochastic heat equation in dimensions three and higher}
\author{Yu Gu, Jiawei Li}

\address[Yu Gu]{Department of Mathematics, Carnegie Mellon University, Pittsburgh, PA 15213, USA}

\address[Jiawei Li]{Department of Mathematics, Carnegie Mellon University, Pittsburgh, PA 15213, USA}

\maketitle

\begin{abstract}
We study the solution to a nonlinear stochastic heat equation in $d\geq 3$. The equation is driven by a Gaussian multiplicative noise that is white in time and smooth in space. For a small coupling constant, we prove (i) the solution converges to the stationary distribution in large time; (ii) the diffusive scale fluctuations are described by the Edwards-Wilkinson equation.

\bigskip


\medskip

\end{abstract}
\maketitle

\section{Introduction}
\subsection{Main result}
We study the solution to the nonlinear stochastic heat equation 
\begin{equation}
\partial_t u=\Delta u+\beta \sigma(u)\dot{W}_\phi(t,x), \quad\quad t>0, x\in\R^d, d\geq 3,
    \label{eqn}
\end{equation}
with constant initial data $u(0,x)\equiv 1$, where $\beta>0$ is a constant. We assume that $\sigma(\cdot)$ is a global Lipschitz function satisfying $|\sigma(x)-\sigma(y)|\leq \sigL |x-y|$ for all $x,y\in\R^d$. Here $\sigL$ is a fixed positive constant. Moreover, $\dot{W}_\phi$ is a centered Gaussian noise that is white in time and smooth in space, constructed from a spacetime white noise $\dot{W}$ and a non-negative mollifier $\phi\in C_c^\infty(\R^d)$:
\[
\dot{W}_\phi(t,x)=\int_{\R^d}\phi(x-y)\dot{W}(t,y)\dd y.
\]
The covariance function is given by 
\[
\begin{aligned}
&\E[\dot{W}_\phi(t,x)\dot{W}_\phi(s,y)]=\delta_0(t-s)R(x-y), \\
&R(x)=\int_{\R^d}\phi(x+y)\phi(y)\dd y \in C_c^\infty(\R^d).
\end{aligned}
\]
 Under our assumptions, there exists a unique continuous random field as the mild solution  to \eqref{eqn}, given by
\begin{equation}
    u(t,x)=1+\beta\int^t_0\int_{\R^d}p(t-s,x-y)\sigma(u(s,y)) \dot{W}_\phi(s,y)\dd y \dd s,
    \label{e.mild}
\end{equation} 
where $p(t,x)=(4\pi t)^{-d/2}e^{-\frac{\lvert x\rvert^2}{4t}}$ is the heat kernel, and the stochastic integral in \eqref{e.mild} is interpreted in the It\^o-Walsh sense. We rescale the solution diffusively, and define 
\[u_\varepsilon(t,x)=u\left(\frac{t}{\varepsilon^2},\frac{x}{\varepsilon}\right).
\]
The first result is about the behavior of $u$ as $t\to\infty$.
\begin{theorem}
There exists $\beta_0=\beta_0(d,\phi,\sigma)>0$ such that if $\beta<\beta_0$, then $u(t,\cdot)\Rightarrow Z(\cdot)$ in $C(\R^d)$, as $t\to\infty$, where $Z(\cdot)$ is a stationary random field.
\label{t.thm1}
\end{theorem}

On top of this result, we obtain the Edwards-Wilkinson limit as follows:

\begin{theorem}
Under the same assumption of Theorem~\ref{t.thm1}, for any test function $g\in C_c^\infty(\R^d)$ and $t>0$, we have
\[\frac{1}{\varepsilon^{\frac{d}{2}-1}}\int_{\R^d}\left(u_\varepsilon(t,x)-1\right)g(x)\dd x\Rightarrow\int_{\R^d}\mathcal{U}(t,x)g(x)\dd x
\]
in distribution as $\eps\to0$, where $\mathcal{U}$ solves the Edwards-Wilkinson equation 
\[
\partial_t \U=\Delta\mathcal{U}+\beta \nu_\sigma \dot{W}(t,x),\quad\quad \U(0,x)\equiv 0,
\]
and $\nu_\sigma$ is the effective constant depending on $\sigma$, the spatial covariance function $R$, and the stationary random field $Z$ obtained in Theorem~\ref{t.thm1}:
\begin{equation}
\nu_\sigma^2=\int_{\R^d}\E[\sigma(Z(0))\sigma(Z(x))]R(x)\dd x.
\label{v_sigma}
\end{equation}
\label{t.thm2}
\end{theorem}

\subsection{Context}
The linear version of \eqref{eqn} was studied in \cite{cosco,GRZ17,nikos,mukherjee2016weak}: for small $\beta$ and the equation
\begin{equation}\label{e.linearshe}
\partial_t u=\Delta u+\beta u\dot{W}_\phi(t,x),
\end{equation}
results similar to Theorem~\ref{t.thm1} and \ref{t.thm2} were proved: (i) the pointwise distribution of $u(t,x)$ converges as $t\to\infty$; (ii) as a random Schwartz distribution, $\eps^{1-\frac{d}{2}}[u_\eps(t,\cdot)-1]$ converges to the Gaussian field given by the solution to the Edwards-Wilkinson equation. Through a Hopf-Cole transformation $h=\log u$, a KPZ-type of equation 
\begin{equation}\label{e.kpz}
\partial_t h=\Delta h+|\nabla h|^2+\beta \dot{W}_\phi(t,x)
\end{equation}
was also studied, and the same Edwards-Wilkinson limit was established in \cite{cosco,dunlap2018kpz,nikos,magnen2017diffusive}, see also \cite{comets}. Similar results were proved in \cite{caravenna2015universality,CRS18,chatterjee2018constructing,kpz2} when $d=2$, where the coupling constant $\beta$ is tuned logarithmically in $\eps$. The previous studies of the nonlinear equation \eqref{e.kpz}   
all rely on the Hopf-Cole transformation and the fact that the solution to the linear equation \eqref{e.linearshe} can be written explicitly by the Feynman-Kac formula or the Wiener chaos expansion. In light of the Hairer-Quastel universality result in the subcritical setting \cite{hairer2018class}, 
it is very natural to ask that, in the present critical setting, if we can study a more general   Hamilton-Jacobi equation 
\begin{equation}\label{e.hjb}
\partial_t h=\Delta h+H(\nabla h)+\beta \dot{W}_\phi(t,x),
\end{equation}
where the Hamiltonian $H$ is not necessarily quadratic, and prove a similar result of convergence to the Edwards-Wilkinson equation, for small $\beta$. The only result in this direction that we are aware of is a two-dimensional anisotropic KPZ equation studied in \cite{cannizzaro20192d}, where the authors considered the nonlinearity $H(\nabla h)=(\partial_{x_1}h)^2-(\partial_{x_2} h)^2$ and proved the existence of subsequential limits of the solutions started from an invariant measure.

In this short note, we study \eqref{eqn}, which to some extent sits between the linear equation \eqref{e.linearshe} and the nonlinear equation \eqref{e.hjb}. The nonlinear term $\sigma(u)$ excludes the use of the Feynman-Kac formula or the Wiener chaos expansion as in the case of \eqref{e.linearshe}, so the previous approaches do not apply. Meanwhile,  \eqref{eqn} is less nonlinear compared to \eqref{e.hjb}, and we are able to make a substantial use of the mild formulation \eqref{e.mild}.

Part of our approach is inspired by another line of work, where similar results were proved for the spatial averages of $u(t,\cdot)$ 
\cite{huang2018central,Huang2019,nualart2019averaging}. For a large class of equations and noises, which in particular covers \eqref{eqn}, central limit theorems were proved for the random variables 
\[
\eps^{-\frac{d}{2}}\int_{\R^d} [u(t,\tfrac{x}{\eps})-1]g(x)\dd x.
\]
 Studying the scaling $(t,x)\mapsto (\tfrac{t}{\eps^2},\tfrac{x}{\eps})$ as in our case requires a good understanding of the local statistics of $u(t,x)$ as $t\to\infty$, and this is provided by Theorem~\ref{t.thm1} by proving the convergence to a stationary distribution. The local statistical property of $Z(\cdot)$ appears naturally in the expression of the effective variance \eqref{v_sigma}, see the heuristic argument at the beginning of Section~\ref{s.proof2}. 
 
 For the linear stochastic heat equation \eqref{e.linearshe} in $d\geq3$ with $\beta\ll1$, the convergence to the stationary solution was shown in \cite{dunlap2018random,kifer1997burgers}, based on the Feynman-Kac formula. For semilinear equations, the existence of stationary solutions/invariant measures was proved e.g. in the early work \cite{da1992invariant,tessitore1998invariant}, but the convergence to the invariant measure as stated in Theorem~\ref{t.thm1} seems to be unknown. Although our main focus of the paper is on the constant initial data, a similar proof works for more general cases. In Remark~\ref{r.624} below, we explain how to adapt the proof of Theorem~\ref{t.thm1} to cover the example of  ``small'' perturbations of the constant initial data.
 
 It is worth mentioning that the assumption of small $\beta$ is necessary for the result to hold. We know from \cite{mukherjee2016weak} that the pointwise distribution of $u(t,x)$ converges to zero as $t\to\infty$, if $\beta$ is beyond a critical value. The recent works \cite{cosco,nikos} extend the result in Theorem~\ref{t.thm2} to the whole regime of $\beta$ in which $\nu_{\mathrm{\sigma}}<\infty$, in the linear case of $\sigma(x)\equiv x$.

\subsection*{Organization of the paper}
In Section~\ref{s.pre}, we introduce the basic tools of analysis on Gaussian space and prove some estimates on the solution $u$ as well as its Malliavin derivative that are used in the sequel. The proofs of Theorems~\ref{t.thm1} and \ref{t.thm2} are in Sections~\ref{s.proof1} and \ref{s.proof2} respectively.

\subsection*{Notations} We use the following notations and conventions throughout this paper.

(i) We use $a\les b$ to denote $a\leq Cb$ for some constant $C$ that is independent of $\eps$. For instance, as $\sigma$ is global Lipschitz, we have $|\sigma(x)|\les 1+|x|$.


(iii) $\|\cdot\|_p$ denotes the $L^p(\Omega)$ norm of the probability space $(\Omega,\mathscr{G},\mathbb{P})$ where the spacetime white noise $\dot{W}$ is built on.

(iv) $p(t,x)=(4\pi t)^{-d/2}e^{-\frac{\lvert x\rvert^2}{4t}}$ is the heat kernel of $\partial_t-\Delta$.

(v) The Fourier transform of $f$ is denoted by $\hat{f}(\xi)=\int_{\R^d} f(x)e^{-i\xi\cdot x} \dd x$.

\subsection*{Acknowledgement}
We thank the two anonymous referees for  many helpful suggestions to improve the presentation.
The work was partially supported by the NSF through DMS-1907928 and the Center for
Nonlinear Analysis of CMU.

\section{Preliminaries}
\label{s.pre}
Throughout this paper, we consider the centered Gaussian noise $\dot{W}_\phi(t,x)$ on $\R\times \R^d$ with $d\geq 3$, whose covariance is given by
\[\E[\dot{W}_\phi(t,x)\dot{W}_\phi(s,y)]=\delta_0(t-s)R(x-y),\]
where the spatial covariance function $R$ is assumed to be smooth and has a compact support. One may associate an isonormal Gaussian process to this noise. Consider a stochastic process \[
\left\{W_\phi(h)=\int_{\R^{1+d}} h(s,x)\dot{W}_\phi(s,x)\dd x\dd s,\quad h\in C_c^\infty(\R\times \R^d)\right\}
\] defined on a complete probability space $(\Omega,\mathscr{G},\mathbb{P})$ satisfying
\begin{align*}
    \E[W_\phi(h)W_\phi(g)]=\int^\infty_{-\infty}\int_{\R^{2d}}h(s,x)g(s,y)R(x-y)\dd x\dd y\dd s.
\end{align*}
As $R$ is positive definite, the above integral defines an inner product, which we denote by $\langle \cdot,\cdot \rangle_{\mathcal{H}}$, so that $\E[W_\phi(h)W_\phi(g)]=\langle h,g\rangle_{\mathcal{H}}$ for all $h,g\in C_c^\infty(\R\times \R^d)$. Complete the space $C_c^\infty(\R\times \R^d)$ with respect to this inner product, and denote the completion by $\mathcal{H}$, and thus we obtain an isonormal Gaussian process $\{W_\phi(h),h\in\mathcal{H}\}$. Consider the $\sigma$-algebra defined by
\[
\mathscr{F}_t^0:=\sigma\{W_\phi(\mathds{1}_{[0,s]}(\cdot)\mathds{1}_A(\cdot)):0\leq s\leq t, A\in \mathscr{B}_b(\R^d)\},
\]
where $\mathscr{B}_b(\R^d)$ denotes the bounded Borel subsets of $\R^d$ and let $\mathscr{F}_t$ denote the completion of $\mathscr{F}_t^0$ with respect to the measure $\mathbb{P}$. Denote $\mathscr{F}=\{\mathscr{F}_t:t\geq 0\}$, which is the natural filtration generated by $\dot{W}_\phi$, and then for all $\mathscr{F}$-adapted, jointly measurable random field $\{\Phi(t,x):(t,x)\in \R\times \R^d\}$ such that
\[\E[\lVert \Phi\rVert_\mathcal{H}^2]=\E\left[\int^\infty_{-\infty}\int_{\R^{2d}}\Phi(t,x)\Phi(t,y)R(x-y)\dd x\dd y\dd t\right]<\infty,\]
the stochastic integral 
\[\int^\infty_{-\infty}\int_{\R^d} \Phi(t,x)\dd W_\phi(t,x)\]
is well-defined in the It{\^o}-Walsh sense, and the It{\^o} isometry holds:
\begin{equation}
    \E\left[\left\vert \int^\infty_{-\infty}\int_{\R^d} \Phi(t,x)\dd W_\phi(t,x)\right\vert^2\right]=\E[\lVert \Phi\rVert_\mathcal{H}^2].
    \label{walsh_isometry}
\end{equation}
Throughout the paper, we will not distinguish the following two expressions: 
\[
\int \Phi(t,x)\dot{W}_\phi(t,x)\dd x\dd t\quad \mbox{ and } \quad \int \Phi(t,x)\dd W_\phi(t,x).
\]

In the proof of Theorem~\ref{t.thm2}, we also need to adopt methods from Malliavin calculus, so let us introduce a differential structure on the infinite-dimensional space in the manner of Malliavin. We shall follow the notations from \cite{nualart2006malliavin}. Let $\mathcal{S}$ be the space of random variables of the form $F=f(W_\phi(h_1),\cdots,W_\phi(h_n))$, where $f\in C^\infty(\R^n)$ with all  derivatives having at most polynomial growth. Its Malliavin derivative is an $\mathcal{H}$-valued random variable given by 
\[DF=\sum_{i=1}^n\partial_i f(W_\phi(h_1),\cdots,W_\phi(h_n))h_i,\]
where $\partial_i f$ denotes the partial derivative of $f$ with respect to the $i$-th variable. By induction, one may define the higher-order derivative $D^l F$, $l=1,2,\cdots$, which is an $\mathcal{H}^{\otimes l}$-valued random variable. Then the Sobolev norm $\lVert \cdot \rVert_{r,p}$ of such an $F$ is defined as
\[
\lVert F\rVert_{r,p}=\left(\E[\lvert F\rvert^p]+\sum_{l=1}^r\E[\lVert D^lF\rVert_{\mathcal{H}^{\otimes l}}^p]\right)^{\frac{1}{p}}.
\]
Complete $\mathcal{S}$ with respect to this Sobolev norm and denote the completion by $\mathbb{D}^{r,p}$.

Let $\delta$ be the divergence operator, which is the adjoint operator of the differential operator $D$. For each $v\in \mathrm{Dom}\delta$, define $\delta(v)$ to be the unique element in $L^2(\Omega)$ such that 
\begin{equation}\label{e.defdivergence}
\E[F\delta(v)]=\E[\langle DF,v\rangle_\mathcal{H}],\quad \forall F\in \mathbb{D}^{1,2}.
\end{equation}

For $v\in\mathrm{Dom}\delta$, $\delta(v)$ is also called the Skorokhod integral of $v$. In our case of $v$ being adapted to the filtration $\mathcal{F}_t$, it coincides with the Walsh integral, which is written as $\delta(v)=\int v(t,x)\dd W_\phi(t,x)$. The Malliavin derivative $D_{t,x}\delta(v)=D\delta(v)(t,x)$ is given by (see Proposition \blue{1.3.8}, Chapter 1, \cite{nualart2006malliavin})
\[D_{t,x}\delta(v)=v(t,x)+\int^\infty_{-\infty}\int_{\R^d}D_{t,x}v(s,y)\dd W_\phi(s,y).\]

Using the mild formulation \eqref{e.mild} of the nonlinear stochastic heat equation, we may write
\begin{align*}
    u(t,x)=1+\delta(v_{t,x}),\vphantom{\int^1_0}
\end{align*}
where 
\[v_{t,x}(s,y)=\beta \mathds{1}_{[0,t]}(s)p(t-s,x-y)\sigma(u(s,y)),\]
so the Malliavin derivative of the solution $u$ is given by 
\[
\begin{aligned}
D_{r,z}u(t,x)=&\beta\mathds{1}_{[0,t]}(r)p(t-r,x-z)\sigma(u(r,z))\\
&+\beta\int^t_r\int_{\R^d}p(t-s,x-y)\Sigma(s,y)D_{r,z}u(s,y)\dd W_\phi(s,y).
\end{aligned}
\]
where $\Sigma(s,y)$ is an adapted process, bounded by the Lipschitz constant $\sigL$. If we further assume that $\sigma(\cdot)$ is continuously differentiable, then $\Sigma(s,y)=\sigma'(u(s,y))$.

We first prove the following moment estimates on $u$. This result is very important as it will be applied in the proofs of both Theorem \ref{t.thm1} and Theorem \ref{t.thm2}.
\begin{lemma}\label{l.mmbd}
For any $p>1$, there exists $\beta_0=\beta_0(d,p,\phi,\sigL)$ and $C=C(\beta_0,d,p,\phi,\sigL)$ such that if $\beta<\beta_0$, we have
\[
\sup_{t\geq 0}\|u(t,0)\|_p\leq C.
\]
\end{lemma}

\begin{proof}
By the mild formulation, 
\[
u(t,0)=1+\beta\int_0^t \int_{\R^d}p(t-s,y)\sigma(u(s,y))\dd W_\phi(s,y).
\]
For any $n\in \Z_{\geq0}$, by the Burkholder-Davis-Gundy inequality and the Minkowski inequality, we have
\[
\begin{aligned}
\|u(t,0)\|_{2n}^{2n}\les& 1+\beta^{2n}\left\Vert\int_0^t\int_{\R^{2d}} p(t-s,y_1)p(t-s,y_2)\sigma(u(s,y_1))\sigma(u(s,y_2)) R(y_1-y_2)\dd y_1\dd y_2\dd s\right\Vert_n^n\\
\leq&1+\beta^{2n}\left(\int_0^t \int_{\R^{2d}} p(t-s,y_1)p(t-s,y_2)\|\sigma(u(s,y_1))\sigma(u(s,y_2))\|_nR(y_1-y_2)\dd y_1\dd y_2 \dd s\right)^n.
\end{aligned}
\]
Applying H\"older's inequality and using the stationarity of $u(s,\cdot)$, we further obtain
\[
\|\sigma(u(s,y_1))\sigma(u(s,y_2))\|_n\leq \|\sigma(u(s,0))\|_{2n}^2\les 1+\|u(s,0)\|_{2n}^2.
\]
Thus, if we define $f(t)=\|u(t,0)\|_{2n}^2$, then 
\[
f(t)\les 1+\beta^2\int_0^t\int_{\R^{2d}} p(t-s,y_1)p(t-s,y_2)  (1+f(s)) R(y_1-y_2)\dd y_1\dd y_2 \dd s.
\]
For the integration in $y_1,y_2$, we use the elementary inequality
\[
\int_{\R^{2d}} p(t-s,y_1)p(t-s,y_2) R(y_1-y_2)\dd y_1\dd y_2 \les 1\wedge (t-s)^{-d/2},
\]
which yields the integral inequality for $f$: there exists $C>0$ independent of $t$ such that 
\begin{equation}\label{e.291}
f(t)\leq C+C\beta^2\int_0^t [1\wedge (t-s)^{-d/2}] f(s)\dd s.
\end{equation}
As the kernel $1\wedge s^{-d/2}$ is in $L^1(\R_+)$ in $d\geq 3$, we choose $\beta$ small so that 
\[
C\beta^2 \int_0^\infty [1\wedge s^{-d/2}] \dd s<1
\]
and a direct iteration of \eqref{e.291} shows that $\sup_{t\geq0} f(t) \les 1$, which completes the proof.
\end{proof}

Next, we establish an upper bound of $\lVert D_{r,z}u(t,x)\rVert_p$ which will be useful in the proof of Theorem \ref{t.thm2}. 

\begin{lemma}
For all $t>0$, $x\in\R^d$ and $p\geq 2$, there exists some constant $C=C(\beta,d,p,\phi,\sigL)$ such that 
\[\lVert D_{r,z}u(t,x)\rVert_p\leq Cp(t-r,x-z),\quad \mbox{ for all } (r,z)\in(0,t)\times\R^d. \]
\label{l.bdde}
\end{lemma}
\begin{proof}
The proof follows from \cite[Lemma 3.11]{chen2019regularity}, with slight modifications. 
Let $S_{r,z}(t,x)$ be the solution to 
\[S_{r,z}(t,x)=\beta p(t-r,x-z)+\beta \int^t_r\int_{\R^d}p(t-s,x-y)\Sigma(s,y)S_{r,z}(s,y)\dd W_\phi(s,y),\]
then due to the uniqueness of the solution to the above equation, 
\[D_{r,z}u(t,x)= S_{r,z}(t,x)\sigma(u(r,z)).\]
By the Burkholder-Davis-Gundy inequality, 
\begin{align*}
    \lVert S_{r,z}(t,x)\rVert_p^2\les& \beta^2p(t-r,x-z)^2+\beta^2\int^t_r\int_{\R^{2d}}p(t-s,x-y_1)p(t-s,x-y_2)\\
   &\cdot \lVert S_{r,z}(s,y_1)\rVert_p\lVert S_{r,z}(s,y_2)\rVert_p R(y_1-y_2)\dd y_1\dd y_2 \dd s.
\end{align*}
Notice that if we set $t=\theta+r$ and $x=\eta+z$, then the above estimate for $S_{r,z}(\theta+r,\eta+z)$ is independent of $(r,z)$, so for convenience we may denote \[g(\theta,\eta)=\lVert S_{r,z}(\theta+r,\eta+z)\rVert_p,\]
and thus
\begin{align*}
    g(\theta,\eta)^2\leq & C\beta^2p(\theta,\eta)^2+C\beta^2\int^\theta_0\int_{\R^{2d}}p(\theta-s,\eta-y_1)p(\theta-s,\eta-y_2)\\
    &\cdot g(s,y_1)g(s,y_2)R(y_1-y_2)\dd y_1\dd y_2 \dd s.
\end{align*}
Then according to Lemma 2.7 in \cite{chen2019regularity}, 
we may conclude that 
\[g(\theta,\eta)\leq \sqrt{C\beta^2}p(\theta,\eta)H(\theta,2C\beta^2)^{\frac{1}{2}},\]
and $H$ is defined as
\[H(t,\lambda)=\sum_{n=0}^\infty \lambda^n h_n(t),\]
where $h_0(t)=1$ and
\[h_n(t)=\int^t_0h_{n-1}(s)\int_{\R^d}p(t-s,z)R(z)\dd z\dd s,\quad n\geq 1.\]
We notice that for all $t$,
\[\lvert h_1(t)\rvert \leq  \int_0^\infty\int_{\R^d} p(s,z)R(z)\dd z\dd s=:C_R,
\]
and thus for all $n\geq1$ and $t\geq0$,
\[\lvert h_n(t)\rvert\leq C_R^n.
\]
Therefore, for sufficiently small $\beta$ 
we have $H(\theta,2C\beta^2)\les1$ uniformly in $\theta$, and hence, it follows that for all $p>1$, 
\[\lVert D_{r,z}u(t,x)\rVert_p\leq \lVert S_{r,z}(t,x)\rVert _{2p}\lVert \sigma(u(r,z))\rVert_{2p}\leq Cp(t-r,x-z),\]
where in the last step we used the fact that $|\sigma(x)|\les 1+|x|$ and Lemma~\ref{l.mmbd}.
\end{proof}

%

In the proof of Theorem~\ref{t.thm2}, we need the following result which is also used in  \cite{huang2018central}. With the help of the this result, to establish the required convergence in law, we only need to control the $L^2$-distance on the right-hand side of the result below, which is more amenable to calculations. 
\begin{proposition}
Let $X$ be a random variable such that $X=\delta(v)$ for $v\in\mathrm{Dom}\,\delta $. Assume $X\in\mathbb{D}^{1,2}$. Let $Z$ be a centered Gaussian random variable with variance $\Sigma$. For any $C^2$-function $h:\R\to\R$ with bounded second order derivative, then
\[\lvert \E h(X)-\E h(Z)\rvert\leq \frac{1}{2}\lVert h''\rVert_\infty\sqrt{\E\left[\left(\Sigma-\langle DX, v\rangle_{\mathcal{H}}\right)^2\right]}.\]
\label{p.distribution}
\end{proposition}

We also need to apply the following version of Clark-Ocone formula in the proof of Theorem \ref{t.thm2}, to estimate certain covariance.
\begin{proposition}[Clark-Ocone Formula]
Let $X\in \mathbb D^{1,2}$, then 
\[
X=\mathbb{E}[X]+\int_{\R_+\times \mathbb{R}^d}\mathbb{E}[D_{r,z}X|\mathscr{F}_r]dW_\phi (r,z).\]
\label{p.clark-ocone}
\end{proposition}
The proof of this formula can be found e.g. in \cite[Proposition 6.3]{chenle}.

\section{Proof of Theorem~\ref{t.thm1}}
\label{s.proof1}
To prove the convergence of $u(t,\cdot)$ to the stationary distribution, we use a rather standard approach: instead of sending $t\to\infty$ and showing $u(t,\cdot)$ converges \emph{weakly}, we initiate the equation at $t=-K$ and consider the solution at $t=0$, then send $K\to\infty$ to prove the \emph{strong convergence}. 

More precisely, for $K>0$, we consider a family of equations indexed by $K$:
\begin{equation}
\begin{aligned}
  &\partial_t u_K=\Delta u_K+\beta \sigma(u_K)\dot{W}_\phi(t,x), \quad\quad t>-K, x\in\R^d,\\
  &u_K(-K,x)\equiv1.
    \label{shifted_eqn}
    \end{aligned}
\end{equation}
By the stationarity of the noise $\dot{W}_\phi$, we know that for all $K>0$, $u(K,\cdot)$ and $u_{K}(0,\cdot)$ as random variables taking values in $C(\R^d)$, have the same law.
Then the problem reduces to proving the weak convergence of $C(\R^d)$-valued random variables $u_K(0,\cdot)$. 

The following two propositions combine to complete the proof of Theorem~\ref{t.thm1}.
\begin{proposition}\label{p.cauchy}
For each $x\in\R^d$, $\{u_K(0,x)\}_{K\geq0}$ is a Cauchy sequence in $L^2(\Omega)$.
\end{proposition}
\begin{proposition}\label{p.tight}
The sequence of $C(\R^d)$-valued random variables $\{u(t,\cdot)\}_{t\geq0}$ is tight.
\end{proposition}

Since $u_K(0,x)$ is a stationary process in $x$, to show $\{u_K(0,x)\}_{K\geq0}$ is Cauchy in $L^2(\Omega)$, we only need to consider $x=0$. We write \eqref{shifted_eqn} in the mild formulation:
\begin{align}
    u_K(t,x)&=1+\beta\int^t_{-K}\int_{\R^d}p(t-s,x-y)\sigma(u_K(s,y))\dd W_\phi(s,y), \quad t>-K, x\in\R^d.
    \label{mild u_k}
\end{align}
Therefore, for any $K_1>K_2\geq 0$, and $t\geq -K_2$, we can write the difference as 
\begin{equation}\label{e.292}
u_{K_1}(t,0)-u_{K_2}(t,0)=\beta[I_{K_1,K_2}(t)+J_{K_1,K_2}(t)],
\end{equation}
with 
\[
I_{K_1,K_2}(t)=\int_{-K_1}^{-K_2}\int_{\R^d} p(t-s,y)\sigma(u_{K_1}(s,y))\dd W_\phi(s,y),
\]
\[
J_{K_1,K_2}(t)= \int_{-K_2}^t \int_{\R^d} p(t-s,y)[\sigma(u_{K_1}(s,y))-\sigma(u_{K_2}(s,y))]\dd W_\phi(s,y).
\]
For $t>-K_2>-K_1$, define 
\begin{equation}\label{e.defalpha}
\alpha_{K_1,K_2}(t)=\int_{t+K_2}^{t+K_1} (1\wedge s^{-d/2})\dd s.
\end{equation}
and
\begin{equation}\label{e.defgamma}
\gamma_{K_1,K_2}(t)=\E[|u_{K_1}(t,0)-u_{K_2}(t,0)|^2].
\end{equation}
We have the following lemmas.
\begin{lemma}\label{l.lemI}
For any $t>-K_2$, $\E[|I_{K_1,K_2}(t)|^2]\les \alpha_{K_1,K_2}(t).$
\end{lemma}

\begin{lemma}\label{l.lemJ}
$\E[|J_{K_1,K_2}(t)|^2]\les \int_{-K_2}^t [1\wedge (t-s)^{-d/2}]\gamma_{K_1,K_2}(s)\dd s$
\end{lemma}

\begin{proof}[Proof of Lemma~\ref{l.lemI}]
By It\^o's isometry and the fact that $|\sigma(x)|\les 1+|x|$, we have 
\[
\begin{aligned}
\E[|I_{K_1,K_2}(t)|^2] \les  \int_{-K_1}^{-K_2} \int_{\R^{2d}} &p(t-s,y_1)p(t-s,y_2)\\
&\cdot \E[(1+u_{K_1}(s,y_1))(1+u_{K_1}(s,y_2))]R(y_1-y_2)\dd y_1\dd y_2 \dd s.
\end{aligned}
\]
Further applying Lemma~\ref{l.mmbd} yields
\[
\begin{aligned}
\E[|I_{K_1,K_2}(t)|^2] \les & \int_{-K_1}^{-K_2} \int_{\R^{2d}} p(t-s,y_1)p(t-s,y_2) R(y_1-y_2)\dd y_1\dd y_2 \dd s\\
\les & \int_{-K_1}^{-K_2}[1\wedge (t-s)^{-d/2}]\dd s= \alpha_{K_1,K_2}(t).
\end{aligned}
\]
\end{proof}

\begin{proof}[Proof of Lemma~\ref{l.lemJ}]
By It\^o's isometry, the Lipchitz property of $\sigma$, and the stationarity of $u_K(s,\cdot)$, we have
\[
\begin{aligned}
\E[|J_{K_1,K_2}(t)|^2]\les \int_{-K_2}^t \int_{\R^d} &p(t-s,y_1)p(t-s,y_2) \\
&\cdot \E[|u_{K_1}(s,0)-u_{K_2}(s,0)|^2] R(y_1-y_2)\dd y_1 \dd y_2 \dd s.
\end{aligned}
\]
After an integration in $y_1,y_2$, the right-hand side of the above inequality can be bounded by 
\[
\int_{-K_2}^t [1\wedge (t-s)^{-d/2}] \E[|u_{K_1}(s,0)-u_{K_2}(s,0)|^2]  \dd s,
\]
which completes the proof.
\end{proof}

Combining the above two lemmas with \eqref{e.292}, we have the integral inequality
\begin{equation}\label{e.inineq}
\begin{aligned}
\gamma_{K_1,K_2}(t)=&\beta^2 \E[|I_{K_1,K_2}(t)+J_{K_1,K_2}(t)|^2]\\
\leq & C\beta^2 \alpha_{K_1,K_2}(t)+C\beta^2 \int_{-K_2}^t [1\wedge (t-s)^{-d/2}]\gamma_{K_1,K_2}(s)\dd s, \quad \mbox{ for all } t>-K_2,
\end{aligned}
\end{equation}
where $C>0$ is a constant independent of $t,K_1,K_2$.
The following lemma completes the proof of Proposition~\ref{p.cauchy}.
\begin{lemma}\label{l.624}
For fixed $t>-K_2$, $\gamma_{K_1,K_2}(t)\to0$ as $K_2\to\infty$.
\end{lemma}

\begin{proof}
To ease the notation in the proof, we simply write \eqref{e.inineq} as 
\begin{equation}\label{e.294}
\gamma(t)\leq C\beta^2 \alpha(t)+C\beta^2 \int_{-K_2}^t k(t-s)\gamma(s)\dd s,
\end{equation}
where we omitted the dependence on $K_1,K_2$ and denote $k(s)=1\wedge s^{-d/2}$. 

Before going to the iteration, we claim the following inequality holds:
\begin{equation}\label{e.293}
\int_{-K_2}^t k(t-s)[1\wedge (s+K_2)^{-(\frac{d}{2}-1)}]\dd s\les 1\wedge (t+K_2)^{-(\frac{d}{2}-1)}, \quad \mbox{ for } t\geq -K_2.
\end{equation}
By a change of variable, we have 
\[
\begin{aligned}
&\int_{-K_2}^t k(t-s)[1\wedge  (s+K_2)^{-(\frac{d}{2}-1)}]\dd s\\
&=\int_0^{t+K_2}[1\wedge (t+K_2-s)^{-\frac{d}{2}}][1\wedge s^{-(\frac{d}{2}-1)}]\dd s.
\end{aligned}
\]
First note that the integral is bounded uniformly in $t+K_2$, then we decompose the integration domain: 
\[
\begin{aligned}
&\int_0^{\frac{t+K_2}{2}}[1\wedge (t+K_2-s)^{-\frac{d}{2}}][1\wedge s^{-(\frac{d}{2}-1)}]\dd s\\
&\les \int_0^{\frac{t+K_2}{2}}[1\wedge (t+K_2-s)^{-\frac{d}{2}}]\dd s \les (t+K_2)^{-(\frac{d}{2}-1)},
\end{aligned}
\]
and
\[
\begin{aligned}
&\int_{\frac{t+K_2}{2}}^{t+K_2}[1\wedge (t+K_2-s)^{-\frac{d}{2}}][1\wedge s^{-(\frac{d}{2}-1)}]\dd s\\
&\les (t+K_2)^{-(\frac{d}{2}-1)}\int_{\frac{t+K_2}{2}}^{t+K_2}[1\wedge (t+K_2-s)^{-\frac{d}{2}}]\dd s \les (t+K_2)^{-(\frac{d}{2}-1)},
\end{aligned}
\]
which proves \eqref{e.293}.

Now we iterate \eqref{e.294} to obtain $\gamma(t) \leq \sum_{n=0}^\infty \gamma_n(t)$ with 
\[
\begin{aligned}
&\gamma_0(t)=C\beta^2\alpha(t)\\
&\gamma_n(t)=(C\beta^2)^{n+1}\int_{-K_2<s_n<\ldots<s_1<t} \prod_{j=0}^{n-1} k(s_j-s_{j+1}) \alpha (s_n) \dd s_n\ldots \dd s_1
\end{aligned}
\]
where we used the convention $s_0=t$. From the explicit expression of $\alpha$ in \eqref{e.defalpha}, we know that there exists $C>0$ such that 
\[
\alpha(s)\leq C [1\wedge (s+K_2)^{-(\frac{d}{2}-1)}], \quad \mbox{ for } s>-K_2.
\]
Now we apply \eqref{e.293} to derive that (with a possibly different constant $C>0$)
\[
\gamma_n(t) \leq (C\beta^2)^{n+1} [1\wedge (t+K_2)^{-(\frac{d}{2}-1)}].
\]
Choose $C\beta^2<1$ and sum over $n$, we know that 
\[
\gamma(t) \les 1\wedge (t+K_2)^{-(\frac{d}{2}-1)}\to0 
\]
as $K_2\to\infty$. The proof is complete.
\end{proof}

\begin{remark}\label{r.624}
It is clear from the proof that the assumption of the constant initial data $u(0,x)\equiv u_{K}(-K,x)\equiv1$ can be relaxed. A similar proof should work for more general initial conditions. One particular example for which our proof works is the following. Let $u(0,x)$ be a perturbation of the constant $\lambda>0$ in the sense that $u(0,x)=\lambda+f(x)$ with  $f\in L^1(\R^d)\cap L^\infty(\R^d)$, then \eqref{mild u_k} becomes 
\[
\begin{aligned}
u_{K}(t,x)=&\lambda+\int_{\R^d} p(t+K,x-y)f(y)\dd y\\
&+\beta\int^t_{-K}\int_{\R^d}p(t-s,x-y)\sigma(u_K(s,y))\dd W_\phi(s,y), \quad t>-K, x\in\R^d,
\end{aligned}
\]
with the second term on the r.h.s., which is associated with the initial condition, depending on $K$ as well. By following the same proof as before and the elementary fact that 
\[
\int_{\R^d} p(t+K,x-y)f(y)dy \les  1\wedge (t+K)^{-\frac{d}{2}},
\]
we derive an  integral inequality that is similar to \eqref{e.inineq} (recall that $K_1>K_2\geq 0$ and $t\geq -K_2$)
\[
\begin{aligned}
\tilde{\gamma}_{K_1,K_2}(t)\leq  C[ 1\wedge (t+K_2)^{-\frac{d}{2}}]+C\beta^2\alpha_{K_1,K_2}(t)+C\beta^2\int_{-K_2}^t [1\wedge (t-s)^{-\frac{d}{2}}] \tilde{\gamma}_{K_1,K_2}(s)\dd s\\
\end{aligned}
\]
with $\tilde{\gamma}_{K_1,K_2}(t):=\sup_{x\in\R^d} \E[|u_{K_1}(t,x)-u_{K_2}(t,x)|^2]$. Using the fact that 
\[
1\wedge (t+K_2)^{-\frac{d}{2}} \leq 1\wedge (t+K_2)^{-(\frac{d}{2}-1)},
\]
and applying Lemma~\ref{l.624} again, we conclude the proof.
\end{remark}

\begin{proof}[Proof of Proposition~\ref{p.tight}]
The tightness of $\{u(t,\cdot)\}_{t\geq0}$ in $C(\R^d)$ follows from 

(i) $\{u(t,0)\}_{t\geq0}$ is tight in $\R$; 

(ii) For any $\delta\in(0,1)$ and $n\in \Z_{\geq0}$, there exists a constant $C>0$ such that for $x_1,x_2\in \R^d$ satisfying $|x_1-x_2|\leq 1$ and any $t>0$, 
\begin{equation}\label{e.komo}
\E[|u(t,x_1)-u(t,x_2)|^{2n}]\leq C |x_1-x_2|^{2\delta n}.
\end{equation}

The tightness of $\{u(t,0)\}_{t\geq0}$ comes from the bound $\sup_{t\geq0}\|u(t,0)\|_p\leq C$ given by Lemma~\ref{l.mmbd}. To prove \eqref{e.komo}, we write 
\[
u(t,x_1)-u(t,x_2)=\beta\int_0^t \int_{\R^d} G_{x_1,x_2}(t-s,y)\sigma(u(s,y))\dd W_\phi(s,y).
\]
with
\[
G_{x_1,x_2}(t-s,y)=p(t-s,x_1-y)-p(t-s,x_2-y).
\]
 Follow the same argument in Lemma~\ref{l.mmbd}, we have 
\[
\begin{aligned}
&\E[|u(t,x_1)-u(t,x_2)|^{2n}] \\
&\les \E\left[\left(\int_0^t\int_{\R^{2d}}G_{x_1,x_2}(t-s,y_1)G_{x_1,x_2}(t-s,y_2) \sigma(u(s,y_1))\sigma(u(s,y_2)) 
R(y_1-y_2)\dd y_1\dd y_2 \dd s\right)^n\right]\\
&\les \left(\int_0^t\int_{\R^{2d}}G_{x_1,x_2}(t-s,y_1)G_{x_1,x_2}(t-s,y_2) \|\sigma(u(s,y_1))\sigma(u(s,y_2))\|_n
R(y_1-y_2)\dd y_1\dd y_2 \dd s\right)^n\\
& \les\left(\int_0^t\int_{\R^{2d}}G_{x_1,x_2}(t-s,y_1)G_{x_1,x_2}(t-s,y_2) 
R(y_1-y_2)\dd y_1\dd y_2 \dd s\right)^n.
\end{aligned}
\]
By \cite[Lemma 3.1]{Chen_2019}, for any $\delta \in (0,1)$, there exists a constant $C>0$ such that 
\[
G_{x_1,x_2}(t-s,y) \leq C (t-s)^{-\frac{\delta}{2}}[p(2(t-s),x_1-y)+p(2(t-s),x_2-y)]|x_1-x_2|^{\delta}.
\]
Thus we can bound the integral as 
\[
\begin{aligned}
&\int_0^t\int_{\R^{2d}}G_{x_1,x_2}(t-s,y_1)G_{x_1,x_2}(t-s,y_2) 
R(y_1-y_2)\dd y_1\dd y_2 \dd s\\
&\les  |x_1-x_2|^{2\delta}\sum_{i,j=1,2}\int_0^t \int_{\R^{2d}} (t-s)^{-\delta}p(2(t-s),x_i-y_1)p(2(t-s),x_j-y_2) R(y_1-y_2)\dd y_1\dd y_2 \dd s.
\end{aligned}
\]
For any $i,j$, we write the integral in Fourier domain to derive
\[
\begin{aligned}
&\int_{\R^{2d}} p(2(t-s),x_i-y_1)p(2(t-s),x_j-y_2) R(y_1-y_2) \dd y_1 \dd y_2\\
&= (2\pi)^{-d}\int_{\R^d} e^{-4|\xi|^2(t-s)} \hat{R}(\xi) e^{i\xi\cdot(x_i-x_j)}\dd \xi\les  \int_{\R^d} e^{-4|\xi|^2(t-s)} \hat{R}(\xi) \dd \xi.
\end{aligned}
\]
Another integration in $s$ leads to
\[
\int_0^t\int_{\R^d} (t-s)^{-\delta} e^{-4|\xi|^2(t-s)} \hat{R}(\xi) \dd \xi \dd s \leq \int_0^\infty s^{-\delta} e^{-4s} \dd s  \int_{\R^d} |\xi|^{2\delta-2} \hat{R}(\xi) \dd \xi<\infty,
\]
which implies 
\[
\E[|u(t,x_1)-u(t,x_2)|^{2n}]  \leq C|x_1-x_2|^{2\delta n}
\]
and completes the proof.
\end{proof}

\section{Proof of Theorem~\ref{t.thm2}}
\label{s.proof2}
Recall that the goal is to show that for $g\in C_c^\infty(\R^d)$ and $t>0$,
\[\frac{1}{\varepsilon^{\frac{d}{2}-1}}\int_{\R^d}(u_\varepsilon(t,x)-1)g(x)\dd x\Rightarrow \int_{\R^d}\mathcal{U}(t,x)g(x)\dd x,\]
where $\mathcal{U}$ solves 
\begin{equation}\label{e.limiteq}
\partial_t\mathcal{U}=\Delta\mathcal{U}+\beta \nu_{\sigma}\dot{W}(t,x).
\end{equation}

Before going to the proof, we first give some heuristics which leads to the above equation and the expression of the effective variance 
\begin{equation}\label{e.effvar1}
\nu_{\sigma}^2=\int_{\R^d}\E[\sigma(Z(0))\sigma(Z(x))]R(x)\dd x.
\end{equation}
By the equation satisfied by $u$, we know that the diffusively rescaled solution satisfies
\[
\partial_t u_\eps=\Delta u_\eps+\beta \eps^{-2}\sigma(u_\eps) \dot{W}_\phi(\tfrac{t}{\eps^2},\tfrac{x}{\eps}).
\]
Since $\dot{W}_\phi(t,x)=\int_{\R^d}\phi(x-y)\dot{W}(t,y)\dd y$, using the scaling property of $\dot{W}$, we have
\[
\eps^{-2}\dot{W}_\phi(\tfrac{t}{\eps^2},\tfrac{x}{\eps})\stackrel{\text{law}}{=} \eps^{\frac{d}{2}-1}\dot{W}_{\phi_\eps}(t,x)
\]
as random processes, with $\phi_\eps(t,x)=\eps^{-2-d}\phi(t/\eps^2,x/\eps)$ and 
\[
\dot{W}_{\phi_\eps}(t,x)=\int_{\R^d}\phi_\eps(x-y)\dot{W}(t,y)\dd y.
\]
Then the rescaled fluctuation has the same law as the solution to 
\begin{equation}\label{e.limiteq1}
\partial_t \left(\frac{u_\eps-1}{\eps^{\frac{d}{2}-1}}\right)=\Delta \left(\frac{u_\eps-1}{\eps^{\frac{d}{2}-1}}\right)+\beta \sigma(u_\eps(t,x))\dot{W}_{\phi_\eps}(t,x).
\end{equation}
By Theorem~\ref{t.thm1}, for fixed $t>0$, $\sigma(u_\eps(t,x))=\sigma(u(\tfrac{t}{\eps^2},\tfrac{x}{\eps}))$ has the same local statistical behavior as $\sigma(Z(\tfrac{x}{\eps}))$ when $\eps\ll1$. Since the product between $\sigma(u_\eps)$ and $\dot{W}_{\phi_\eps}$ is  in the It\^o's sense, roughly speaking, these two terms are independent. The fact that $u(t,\cdot)\approx Z(\cdot)$ in law for microscopically large $t$ induces a ``renewal'' mechanism which leads to a $\delta$-correlation in time of the driving force $\sigma(u_\eps(t,x))\dot{W}_{\phi_\eps}(t,x)$ after passing to the limit. While the spatial covariance function of $\dot{W}_{\phi_\eps}$ is $\eps^{-d}R(\tfrac{\cdot}{\eps})$, the overall spatial covariance function is 
\[
\begin{aligned}
&\E[\sigma(u_\eps(t,x))\dot{W}_{\phi_\eps}(t,x)\sigma(u_\eps(t,y))\dot{W}_{\phi_\eps}(t,y)]\\
\approx&\E[\sigma(Z(\tfrac{x}{\eps}))\dot{W}_{\phi_\eps}(t,x)\sigma(Z(\tfrac{y}{\eps}))\dot{W}_{\phi_\eps}(t,y)]\\
=&\E[\sigma(Z(\tfrac{x}{\eps}))\sigma(Z(\tfrac{y}{\eps}))]\E[\dot{W}_{\phi_\eps}(t,x)\dot{W}_{\phi_\eps}(t,y)]\\
=&\E[\sigma(Z(0))\sigma(Z(\tfrac{x-y}{\eps}))]\cdot\eps^{-d} R(\tfrac{x-y}{\eps}).
\end{aligned}
\]
After integrating the variable ``$x-y$'' out, we derive the effective variance  in \eqref{e.effvar1}.


For $t>0$ fixed, by the mild solution formulation
\[u_\varepsilon(t,x)=u\left(\frac{t}{\eps^2},\frac{x}{\eps}\right)=1+\beta\int^{\frac{t}{\varepsilon^2}}_0\int_{\R^d}p\left(\frac{t}{\varepsilon^2}-s,\frac{x}{\varepsilon}-y\right)\sigma(u(s,y))\dd W_\phi(s,y),\]
we may write
\begin{equation}\label{e.defXeps}
\begin{aligned}
    X_\eps=&\frac{1}{\varepsilon^{\frac{d}{2}-1}}\int_{\R^d}(u_\varepsilon(t,x)-1)g(x)\dd x\\
    =&\frac{\beta}{\varepsilon^{\frac{d}{2}-1}}\int_{\R^d}\left(\int^{\frac{t}{\varepsilon^2}}_0\int_{\R^d}p\left(\frac{t}{\varepsilon^2}-s,\frac{x}{\varepsilon}-y\right)\sigma(u(s,y))\dd W_\phi(s,y)\right)g(x)\dd x\\
    =&\delta(v_\eps), 
\end{aligned}
\end{equation}
where we recall that $\delta(\cdot)$ is the divergence operator defined in \eqref{e.defdivergence} and 
\[
    v_\eps(s,y)=\frac{\beta}{\varepsilon^{\frac{d}{2}-1}}\mathds{1}_{[0,\frac{t}{\varepsilon^2}]}(s)\sigma(u(s,y))\int_{\R^d}p\left(\frac{t}{\varepsilon^2}-s,\frac{x}{\varepsilon}-y\right)g(x)\dd x.
    \label{v_i}
\]
As $\U$ solves the equation $\partial_t \U=\Delta \U+\beta \nu_{\sigma} \dot{W}(t,x)$, we have $
\int_{\R^d} \U(t,x)g(x)\dd x 
$ is of Gaussian distribution with zero mean and variance 
\[
\Sigma_g:=\Var\left[\int_{\R^d} \U(t,x)g(x)\dd x \right]=\beta^2\nu_{\sigma}^2\int_0^t \int_{\R^{2d}} p(2(t-s),x_1-x_2) g(x_1)g(x_2) \dd x_1\dd x_2.
\]
Thus, the proof of Theorem~\ref{t.thm2} reduces to showing that
\[
X_\eps=\delta(v_\eps)\Rightarrow N(0,\Sigma_g), \quad \mbox{ as } \eps\to0.
\]
By  Proposition~\ref{p.distribution}, we only need to show 
\begin{equation}\label{e.301}
\E[ |\Sigma_g - \la DX_\eps,v_\eps\ra_{\mathcal{H}}|^2]\to0,\quad \mbox{ as } \eps\to0.
\end{equation}

From \eqref{e.defXeps}, the Malliavin derivative of $X_\eps$ satisfies
\[
\begin{aligned}
    D_{s,y}X_\eps=&v_\eps(s,y)+\frac{\beta}{\varepsilon^{\frac{d}{2}-1}}\int_{\R^d}\left(\int^{\frac{t}{\varepsilon^2}}_s\int_{\R^d}p\left(\frac{t}{\varepsilon^2}-r,\frac{x}{\varepsilon}-z\right)D_{s,y}\sigma(u(r,z))\dd W_\phi(r,z)\right)g(x)\dd x
    \notag\\
    =&v_\eps(s,y)+\frac{\beta}{\varepsilon^{\frac{d}{2}-1}}\int_{\R^d}\left(\int^{\frac{t}{\varepsilon^2}}_s\int_{\R^d}p\left(\frac{t}{\varepsilon^2}-r,\frac{x}{\varepsilon}-z\right)\Sigma(r,z)D_{s,y}u(r,z)\dd W_\phi(r,z)\right)g(x)\dd x,
\end{aligned}
\]
and $\Sigma(r,z)$ as a random variable is bounded by the Lipschitz constant $\sigL$. Recall that 
\[
\langle h,g\rangle_{\mathcal{H}}=\int_{\R^{1+2d}} h(s,x)g(s,y)R(x-y)\dd x\dd y \dd s
\] for all $h,g\in\mathcal{H}$, we have
\[
\langle DX_\eps,v_\eps\rangle_{\mathcal{H}}=\frac{\beta^2}{\varepsilon^{d-2}}(A_{1,\eps}+A_{2,\eps}),
\]
where
\begin{equation}\label{e.defA1}
\begin{aligned}
   A_{1,\eps}=&\int^{\frac{t}{\varepsilon^2}}_0\int_{\R^{2d}}\left(\int_{\R^d}p\left(\frac{t}{\varepsilon^2}-s,\frac{x_1}{\varepsilon}-y_1\right)g(x_1)\dd x_1\right)\\
    &\cdot \left(\int_{\R^d}p\left(\frac{t}{\varepsilon^2}-s,\frac{x_2}{\varepsilon}-y_2\right)g(x_2)\dd x_2\right)\sigma(u(s,y_1))\sigma(u(s,y_2))R(y_1-y_2)\dd y_1\dd y_2\dd s,
\end{aligned}
\end{equation}
and
\begin{equation}\label{e.defA2}
\begin{aligned}
    A_{2,\eps}=&\int^\frac{t}{\varepsilon^2}_0\int_{\R^{2d}}\left(\int^{\frac{t}{\varepsilon^2}}_s\int_{\R^{d}}\left(\int_{\R^d}p\left(\frac{t}{\varepsilon^2}-r,\frac{x_1}{\varepsilon}-z\right)g(x_1)\dd x_1\right)\Sigma(r,z)D_{s,y_1}u(r,z)\dd W_\phi(r,z)\right)\\
    &\cdot \left(\int_{\R^d}p\left(\frac{t}{\varepsilon^2}-s,\frac{x_2}{\varepsilon}-y_2\right)g(x_2)\dd x_2\right)\sigma(u(s,y_2))R(y_1-y_2)\dd y_1\dd y_2\dd s.
\end{aligned}
\end{equation}

We also notice that
\[
\begin{aligned}
\E[ |\Sigma_g - \la DX_\eps,v_\eps\ra_{\mathcal{H}}|^2]&=
\E[|\Sigma_g-\beta^2\varepsilon^{2-d}(A_{1,\eps}+A_{2,\eps})|^2]\\
&\leq 2\lVert \Sigma_g-\beta^2\varepsilon^{2-d}A_{1,\eps}\rVert_2^2+2\beta^4\varepsilon^{4-2d}\lVert A_{2,\eps}\rVert_2^2,
\end{aligned}
\]
so to complete the proof of Theorem~\ref{t.thm2}, it remains to show the right-hand side of the above inequality goes to zero as $\eps\to0$.

\begin{lemma}\label{l.A1}
As $\eps\to0$, $\lVert \Sigma_g-\beta^2\varepsilon^{2-d}A_{1,\eps}\rVert_2\to0$.
\end{lemma}

\begin{lemma}\label{l.A2}
As $\eps\to0$, $\varepsilon^{2-d}\lVert A_{2,\eps}\rVert_2\to0$.
\end{lemma}

In the proofs of Lemma~\ref{l.A1} and \ref{l.A2}, we will use the notation 
\[
g_t(x)=\int_{\R^d} p(t,x-y)g(y)\dd y, \quad \quad t>0, x\in\R^d,
\]
so 
\[
|g_t(x)|\leq \|g\|_{L^\infty(\R^d)},\quad\quad \int_{\R^d}|g_t(x)|\dd x \leq \|g\|_{L^1(\R^d)},
\]
for all $t>0,x\in\R^d$. Without loss of generality, we assume the   function $g$ is non-negative when we  estimate integrals involving $g$.

\begin{proof}[Proof of Lemma~\ref{l.A1}]
We first simplify the expression of $A_{1,\eps}$ defined in \eqref{e.defA1}. After the change of variables $y_1\mapsto y_1/\eps,y_2\mapsto y_2/\eps, s\mapsto s/\eps^2$ and use the scaling property of the heat kernel, we have 
\begin{equation}\label{e.302}
\begin{aligned}
   A_{1,\eps}=&\eps^{-2}\int^{t}_0\int_{\R^{2d}}\left(\int_{\R^d}p\left(t-s,x_1-y_1\right)g(x_1)\dd x_1\right)\\
    &\cdot \left(\int_{\R^d}p\left(t-s,x_2-y_2\right)g(x_2)\dd x_2\right)\sigma\left(u\left(\frac{s}{\eps^2},\frac{y_1}{\eps}\right)\right)\sigma\left(u\left(\frac{s}{\eps^2},\frac{y_2}{\eps}\right)\right)R\left(\frac{y_1-y_2}{\eps}\right)\dd y_1\dd y_2\dd s.
\end{aligned}
\end{equation}
Further change $\tfrac{y_1-y_2}{\eps}\mapsto z$ and $y_2\mapsto y$, we obtain 
\[
\eps^{2-d} A_{1,\eps}=\int_0^t \int_{\R^{2d}} g_{t-s}(y+\eps z) g_{t-s}(y) \sigma\left(u\left(\frac{s}{\eps^2},\frac{y}{\eps}+z\right)\right)\sigma\left(u\left(\frac{s}{\eps^2},\frac{y}{\eps}\right)\right)R(z)\dd y \dd z\dd s,
\]
where we recall that $g_{t-s}(y)=\int_{\R^d} p(t-s,y-z)g(z)\dd z$. The proof is then divided into two steps:

(i) $\beta^2\eps^{2-d}\E[A_{1,\eps}]\to \Sigma_g$ as $\eps\to0$.

(ii) $\eps^{4-2d}\Var[A_{1,\eps}]\to0$ as $\eps\to0$.

To prove (i), it suffices to note that $u(s/\eps^2,x)$ is stationary in $x-$variable, so 
\[
\beta^2 \eps^{2-d}\E[A_{1,\eps}]=\beta^2\int_0^t \int_{\R^{2d}} g_{t-s}(y+\eps z) g_{t-s}(y) \E\left[\sigma\left(u\left(\frac{s}{\eps^2},z\right)\right)\sigma\left(u\left(\frac{s}{\eps^2},0\right)\right)\right]R(z)\dd y \dd z\dd s.
\]
By Theorem~\ref{t.thm1}, we know that for $s>0, z\in\R^d$, the random vector 
\[
(u(s/\eps^2,z),u(s/\eps^2,0))\Rightarrow (Z(z),Z(0))
\]
in distribution as $\eps\to0$. By the fact that $\sigma$ is Lipschitz and applying Lemma~\ref{l.mmbd}, we have the uniform integrability to pass to the limit and conclude that 
\[
\beta^2\eps^{2-d} \E[A_{1,\eps}]\to \beta^2\int_0^t \int_{\R^{2d}} |g_{t-s}(y)|^2\E[\sigma(Z(z))\sigma(Z(0))]R(z)\dd y \dd z\dd s=\Sigma_g.
\]

To prove (ii), we first use \eqref{e.302} to write
\begin{equation}\label{e.varA1}
\begin{aligned}
\eps^{4-2d}\Var[A_{1,\eps}]=&\varepsilon^{-2d}\int^{t}_0\int_{\R^{4d}}g_{t-s}(y_1)g_{t-s}(y_2)g_{t-s}(y_1')g_{t-s}(y_2')\\
    &\cdot\cov[\Lambda_\varepsilon(s,y_1,y_2),\Lambda_\varepsilon(s,y_1',y_2')]R\left(\frac{y_1-y_2}{\varepsilon}\right)R\left(\frac{y_1'-y_2'}{\varepsilon}\right)\dd y_1\dd y_2\dd y_1'\dd y_2'\dd s,
\end{aligned}
\end{equation}
where 
\[\Lambda_\varepsilon(s,y_1,y_2)=\sigma\left(u\left(\frac{s}{\varepsilon^2},\frac{y_1}{\varepsilon}\right)\right)\sigma\left(u\left(\frac{s}{\varepsilon^2},\frac{y_2}{\varepsilon}\right)\right).\]

Applying the Clark-Ocone formula (Proposition \ref{p.clark-ocone}) to $\Lambda_\varepsilon$, we obtain that
\begin{align*}
    \Lambda_\varepsilon(s,y_1,y_2)=\E[\Lambda_\varepsilon(s,y_1,y_2)]+\int^{\frac{s}{\varepsilon^2}}_0\int_{\R^d}\E[D_{r,z}\Lambda_\varepsilon(s,y_1,y_2)\vert\mathscr{F}_r]\dd W_\phi(r,z),
\end{align*}
from which we deduce that
\begin{align*}
    \cov[\Lambda_\varepsilon(s,y_1,y_2),\Lambda_\varepsilon(s,y_1',y_2')]=&\int^{\frac{s}{\varepsilon^2}}_0\int_{\R^{2d}}\E\left[\E[D_{r,z_1}\Lambda_\varepsilon(s,y_1,y_2)\vert \mathscr{F}_r]\E[D_{r,z_2}\Lambda_\varepsilon(s,y_1',y_2')\vert \mathscr{F}_r]\right]\\
    &\cdot R(z_1-z_2)\dd z_1\dd z_2\dd r.\vphantom{\int^1_0}
\end{align*}
By the Chain Rule, we have
\begin{align*}
    D_{r,z}\Lambda_\varepsilon(s,y_1,y_2)=&\Sigma\left(\frac{s}{\eps^2},\frac{y_1}{\eps}\right)D_{r,z}u\left(\frac{s}{\varepsilon^2},\frac{y_1}{\varepsilon}\right)\sigma\left(u\left(\frac{s}{\varepsilon^2},\frac{y_2}{\varepsilon}\right)\right)\\
    &+\Sigma\left(\frac{s}{\eps^2},\frac{y_2}{\eps}\right)D_{r,z}u\left(\frac{s}{\varepsilon^2},\frac{y_2}{\varepsilon}\right)\sigma\left(u\left(\frac{s}{\varepsilon^2},\frac{y_1}{\varepsilon}\right)\right).
\end{align*}
Applying Lemma~\ref{l.mmbd}, \ref{l.bdde}, and using the fact that $\Sigma$ is uniformly bounded, we derive that 
\[
\begin{aligned}
   \lVert \E[D_{r,z}\Lambda_\varepsilon(s,y_1,y_2)\vert \mathscr{F}_r]\rVert_2\leq& \|D_{r,z}\Lambda_\varepsilon(s,y_1,y_2)\|_2 \\
   \les &p\left(\frac{s}{\varepsilon^2}-r,\frac{y_1}{\varepsilon}-z_1\right)+p\left(\frac{s}{\varepsilon^2}-r,\frac{y_2}{\varepsilon}-z_1\right).
\end{aligned}
\]

Therefore, we deduce that
\begin{align*}
    &\left\vert\cov[\Lambda_\varepsilon(s,y_1,y_2),\Lambda_\varepsilon(s,y_1',y_2')]\vphantom{\int^1_0}\right\vert\\
    \les&\int^{\frac{s}{\varepsilon^2}}_0\int_{\R^{2d}}\left(p\left(\frac{s}{\varepsilon^2}-r,\frac{y_1}{\varepsilon}-z_1\right)+p\left(\frac{s}{\varepsilon^2}-r,\frac{y_2}{\varepsilon}-z_1\right)\right)\\
    &\cdot \left(p\left(\frac{s}{\varepsilon^2}-r,\frac{y_1'}{\varepsilon}-z_2\right)+p\left(\frac{s}{\varepsilon^2}-r,\frac{y_2'}{\varepsilon}-z_2\right)\right)R(z_1-z_2)\dd z_1\dd z_2\dd r\\
    = & \sum_{i,j=1,2} F\left(\frac{s}{\eps^2},\frac{y_i-y_j'}{\eps}\right),
\end{align*}
where 
\[
\begin{aligned}
F\left(\frac{s}{\eps^2},\frac{y_i-y_j'}{\eps}\right)=&\int_0^{\frac{s}{\eps^2}} \int_{\R^{2d}} p\left(r,\frac{y_i-y_j'}{\varepsilon}-z_1\right)p\left(r,z_2\right)R(z_1-z_2)\dd z_1\dd z_2\dd r\\
\leq&
\tilde{F}\left(\frac{y_i-y_j'}{\eps}\right),
\end{aligned}
\]
with 
\[
\tilde{F}(x):=\int_0^\infty \int_{\R^{2d}} p(r,x-z_1)p(r,z_2)R(z_1-z_2)\dd z_1\dd z_2\dd r \les 1\wedge |x|^{2-d}.
\]
Going back to \eqref{e.varA1}, it suffices to estimate the integral 
\[
\begin{aligned}
\eps^{-2d}\int^{t}_0\int_{\R^{4d}}&g_{t-s}(y_1)g_{t-s}(y_2)g_{t-s}(y_1')g_{t-s}(y_2')\\
&\cdot \tilde{F}\left(\frac{y_i-y_j'}{\eps}\right)R\left(\frac{y_1-y_2}{\varepsilon}\right)R\left(\frac{y_1'-y_2'}{\varepsilon}\right)\dd y_1\dd y_2\dd y_1'\dd y_2'\dd s
\end{aligned}
\]
for $i,j=1,2$ and show it vanishes as $\eps\to0$. By symmetry, we only need to consider the case $i=j=1$. After a change of variables $y_1\mapsto y_2+\eps y_1, y_1'\mapsto y_2'+\eps y_1'$ and use the fact that $|g_{t-s}(\cdot)|\leq \|g\|_\infty$, the above expression is bounded by 
\[
\begin{aligned}
&\int_0^t \int_{\R^{4d}}g_{t-s}(y_2)g_{t-s}(y_2')R(y_1)R(y_1')\tilde{F}\left(\frac{y_2-y_2'}{\eps}+y_1-y_1'\right)\dd y_1\dd y_2\dd y_1'\dd y_2'\dd s\\
&\les \eps^{d-2}\int_0^t \int_{\R^{4d}}g_{t-s}(y_2)g_{t-s}(y_2')R(y_1)R(y_1')|y_2-y_2'|^{2-d}\dd y_1\dd y_2\dd y_1'\dd y_2'\dd s\les \eps^{d-2}.
\end{aligned}
\]
The proof is complete.
%
\end{proof}

\begin{proof}[Proof of Lemma~\ref{l.A2}]
Define
\begin{equation}\label{e.deftildeA}
\tilde{A}_{2,\eps}(s,x_1,y_1):=\int^{\frac{t}{\varepsilon^2}}_s\int_{\R^d}p\left(\frac{t}{\varepsilon^2}-r,\frac{x_1}{\varepsilon}-z\right)\Sigma(r,z)D_{s,y_1}u(r,z)\dd W_\phi(r,z),
\end{equation}
and we can write
\begin{align*}
    A_{2,\eps}=\int^\frac{t}{\varepsilon^2}_0\int_{\R^{4d}}&\tilde{A}_{2,\eps}(s,x_1,y_1)p\left(
\frac{t}{\varepsilon^2}-s,\frac{x_2}{\varepsilon}-y_2\right)\sigma(u(s,y_2))\\
&
\cdot g(x_1)g(x_2)R(y_1-y_2)\dd x_1\dd x_2\dd y_1\dd y_2\dd s.
\end{align*}
After the change of variable $s\mapsto s/\eps^2, y_i\mapsto y_i/\eps$ and integrating in $x_2$, we have
\[
\begin{aligned}
A_{2,\eps}=\eps^{-d-2}\int_0^t\int_{\R^{3d}}& \tilde{A}_{2,\eps}\left(\frac{s}{\eps^2},x_1,\frac{y_1}{\eps}\right) \sigma\left(u\left(\frac{s}{\eps^2},\frac{y_2}{\eps}\right)\right) g(x_1)g_{t-s}(y_2)\\
& \cdot R\left(\frac{y_1-y_2}{\eps}\right)\dd x_1\dd y_1\dd y_2\dd s.
\end{aligned}
\]
Define 
\[
\begin{aligned}
&B_\eps(s,x_1,x_1',y_1,y_1',y_2,y_2')\\
&=\E\left[\tilde{A}_{2,\eps}\left(\frac{s}{\eps^2},x_1,\frac{y_1}{\eps}\right)\tilde{A}_{2,\eps}\left(\frac{s}{\eps^2},x_1',\frac{y_1'}{\eps}\right)\sigma\left(u\left(\frac{s}{\eps^2},\frac{y_2}{\eps}\right)\right)\sigma\left(u\left(\frac{s}{\eps^2},\frac{y_2'}{\eps}\right)\right)\right],
\end{aligned}
\]
then by Minkowski inequality, we have 
\begin{equation}\label{e.311}
\begin{aligned}
\|A_{2,\eps}\|_2\leq \eps^{-d-2}\int_0^t \bigg(\int_{\R^{6d}}&B_\eps(s,x_1,x_1',y_1,y_1',y_2,y_2')g(x_1)g(x_1')g_{t-s}(y_2)g_{t-s}(y_2')\\
&\cdot R\left(\frac{y_1-y_2}{\eps}\right)R\left(\frac{y_1'-y_2'}{\eps}\right)\dd x_1\dd x_1'\dd y_1\dd y_2\dd y_1'\dd y_2'\bigg)^{\frac12}\dd s.
\end{aligned}
\end{equation}
Meanwhile, by the expression of $\tilde{A}_{2,\eps}$ in \eqref{e.deftildeA} and It\^o's isometry, we have 
\[
\begin{aligned}
B_\eps(s,x_1,x_1',&y_1,y_1',y_2,y_2')\\
=\int^{\frac{t}{\varepsilon^2}}_{\frac{s}{\varepsilon^2}}\int_{\R^{2d}}&\E\left[\Sigma(r,z_1)\Sigma(r,z_2)D_{\frac{s}{\varepsilon^2},\frac{y_1}{\varepsilon}}u(r,z_1)D_{\frac{s}{\varepsilon^2},\frac{y_1'}{\varepsilon}}u(r,z_2)\sigma\left(u\left(\frac{s}{\varepsilon^2},\frac{y_2}{\varepsilon}\right)\right)\sigma\left(u\left(\frac{s}{\varepsilon^2},\frac{y_2'}{\varepsilon}\right)\right)\right]\\
&\cdot p\left(\frac{t}{\varepsilon}-r,\frac{x_1}{\varepsilon}-z_1\right)p\left(\frac{t}{\varepsilon}-r,\frac{x_1'}{\varepsilon}-z_2\right) R(z_1-z_2)\dd z_1\dd z_2\dd r.
\end{aligned}
\]
Applying Lemma~\ref{l.mmbd}, \ref{l.bdde}, Cauchy-Schwarz inequality and a change of variables, 
\[
\begin{aligned}
|B_\eps(s,x_1,x_1',y_1,y_1',y_2,y_2')|\les \eps^{2d-2}\int_s^t\int_{\R^{2d}} &p(r-s,z_1-y_1)p(r-s,z_2-y_1')p(t-r,x_1-z_1)\\
&\cdot p(t-r,x_1'-z_2)R\left(\frac{z_1-z_2}{\eps}\right) \dd z_1 \dd z_2 \dd r.
\end{aligned}
\]
Substitute the above estimate into \eqref{e.311} and integrate in $x_1,x_1'$, we finally obtain
\[
\begin{aligned}
\|A_{2,\eps}\|_2\les\eps^{-3}\int_0^t\bigg(\int_s^t\int_{\R^{6d}}&p(r-s,z_1-y_1)p(r-s,z_2-y_1')g_{t-r}(z_1)g_{t-r}(z_2)g_{t-s}(y_2)g_{t-s}(y_2')\\
&\cdot R\left(\frac{y_1-y_2}{\eps}\right)R\left(\frac{y_1'-y_2'}{\eps}\right)R\left(\frac{z_1-z_2}{\eps}\right)\dd y_1 \dd y_2 \dd y_1' \dd y_2'\dd z_1 \dd z_2 \dd r \bigg)^{\frac12} \dd s.
\end{aligned}
\]
For the integral on the right-hand side of the above inequality, we compute the integral in $y_1,y_2$ explicitly:  
\[
\int_{\R^{2d}}p(r-s,z_1-y_1)g_{t-s}(y_2)R\left(\frac{y_1-y_2}{\eps}\right) \dd y_1 \dd y_2=\eps^d \int_{\R^d} g_{t+r-2s}(z_1-\eps y_1) R(y_1)\dd y_1\les \eps^d.
\]
Similarly, the integral in $y_1',y_2'$ is also bounded by 
\[
\int_{\R^{2d}} p(r-s,z_2-y_1')g_{t-s}(y_2')R\left(\frac{y_1'-y_2'}{\eps}\right) \dd y_1'\dd y_2'\les\eps^d.
\]
Thus, 
\[
\begin{aligned}
\|A_{2,\eps}\|_2\les\eps^{d-3}\int_0^t\bigg(\int_s^t\int_{\R^{2d}}&g_{t-r}(z_1)g_{t-r}(z_2)R\left(\frac{z_1-z_2}{\eps}\right) \dd z_1 \dd z_2 \dd r \bigg)^{\frac12} \dd s\les \eps^{\frac{3d}{2}-3}.
\end{aligned}
\]
The proof is complete.
\end{proof}


\end{document}